 \documentclass[a4paper,12pt]{amsart}
 \usepackage[utf8]{inputenc}
 \usepackage{amsmath,amsthm,verbatim,amssymb,amsfonts,amscd, graphicx,mathrsfs}
 \usepackage{geometry}
 \usepackage{tabularx}
 \usepackage{diagbox}
 \numberwithin{equation}{section}

 \newtheorem{thm}{Theorem}[section]
 \newtheorem{lemma}[thm]{Lemma}
 \newtheorem{cor}[thm]{Corollary}

\theoremstyle{definition}
 \newtheorem{define}[thm]{Definition}
 \newtheorem{rmk}[thm]{Remark}
\newtheorem{question}[thm]{Question}

 \usepackage{enumerate}
 \usepackage{hyperref}
 \hypersetup{colorlinks=true, urlcolor=Cerulean, citecolor=red}
 \usepackage[T1]{fontenc}
 \usepackage{nameref,cleveref}
 \usepackage{nicefrac,xfrac}
 \usepackage[all]{xy}
 \usepackage{color}
\usepackage[section]{placeins}

\begin{document}
\title{Stability of syzygy bundles on varieties of Picard number one}

\author{Chen Jiang}
\address{Shanghai Center for Mathematical Sciences \& School of Mathematical Sciences, Fudan University, Shanghai 200438, China}
\email{chenjiang@fudan.edu.cn}
\author{Peng Ren}
\address{Shanghai Center for Mathematical Sciences,
Fudan University, Shanghai 200438, China}
\email{pren@fudan.edu.cn}


\begin{abstract}
We give a criterion for slope-stability of the syzygy bundle of a globally generated ample line bundle on a smooth projective variety of Picard number $1$ in terms of Hilbert polynomial. As applications, we prove the stability of syzygy bundles on many varieties, such as smooth Fano or Calabi--Yau complete intersections, hyperk\"ahler varieties of Picard number 1, abelian varieties of Picard number $1$, rational homogeneous varieties of Picard number 1, weak Calabi--Yau varieties of Picard number $1$ of dimension $\leq 4$, and Fano varieties of Picard number $1$ of dimension $\leq5$. Also we prove the stability of syzygy bundles on all hyperk\"ahler varieties.
\end{abstract}



\date{\today}
\maketitle

\section{Introduction}
Let $X$ be a smooth projective variety over an algebraically closed field $k$ of characteristic $0$ and let $L$ be a globally generated ample line bundle on $X$. The \emph{syzygy bundle} $M_L$ associated with $L$ is the kernel of the evaluation map of global sections of $L$, and we have the following natural exact sequence
\[0\to M_L\to \mathrm{H}^0(X,L)\otimes \mathcal{O}_X\xrightarrow{\textrm{ev} }L\to0.\]

We are interested in the slope-stability of the syzygy bundle $M_L$. Recall that for a torsion-free sheaf $E$ on $X$, its {\it slope} with respect to $L$ is defined by
$$\mu_L(E):=\frac{c_1(E)L^{\dim X-1}}{\textrm{rk}(E)},$$
and $E$ is {\it $\mu_L$-(semi)stable} if
$\mu_L(F)(\leq)< \mu_L(E)$ for any subsheaf $F\subset E$ with $0<\textrm{rk}(F)<\textrm{rk}(E)$.

Stability of syzygy bundles has been studied extensively.
\begin{enumerate}
 \item When $X$ is a smooth projective curve of genus greater than or equal to $1$, many results are known. In particular, Ein and
Lazarsfeld \cite{ein1992stability} showed that $M_L$ is slope-stable if $\mathrm{deg}(L)\geq 2g+1$.

\item When $X$ is a smooth projective surface, Camere \cite{camere2012stability} proved the stability of syzygy bundles on K3 and abelian surfaces, see also the work of Mukherjee and Raychaudhury \cite{Mukherjee2022}; more generally, Ein, Lazarsfeld, and Mustopa \cite{ein2013stability} showed that $M_{L}$ is slope-stable if $L$ is sufficiently ample. See also \cite{basu2023,2405.17006,  Rekuski23,Torres2022} for related works on stability of syzygy bundles on surfaces.

\item
 In general, Ein, Lazarsfeld, and Mustopa \cite{ein2013stability} showed that when $\textrm{Pic}(X)\cong \mathbb{Z}$, $M_{L}$ is slope-stable if $L$ is sufficiently ample, and they conjectured that this is true for any smooth projective variety without assuming $\textrm{Pic}(X)\cong \mathbb{Z}$, and this conjecture was recently proved by Rekuski \cite{rekuski2024stability}.

\item When $X$ is an abelian variety, Caucci and Lahoz \cite{caucci2021stability} showed that $M_{L}$ is slope-semistable if $L=H^{\otimes d}$ where $H$ is an ample line bundle and $d\geq 2$.

\item See also \cite{brenner2008, 2409.04666,flenner1984restrictions, Macias2011, Miro2023, miro2024, trivedi2010semistability} for related works on stability of syzygy bundles on toric varieties.
\end{enumerate}

Inspired by \cite{caucci2021stability,ein2013stability}, we are interested in the stability of $M_L$ for any globally generated ample line bundle $L$ on $X$. We will give a criterion for stability of syzygy bundles on varieties of Picard number $1$, and then apply the criterion to get stability of syzygy bundles on many varieties, such as Fano or Calabi--Yau complete intersections and hyperk\"ahler varieties.

\subsection{A criterion for stability in terms of Hilbert polynomials}
Recall that for a smooth projective variety $X$ and an ample line bundle $H$ on $X$, by the Riemann--Roch formula, the {\it Hilbert polynomial} $P_H(t)$ with respect to $H$ is a polynomial in $t$ of degree $\dim X$ such that for any integer $m\in \mathbb{Z}$, \[P_H(m)=\chi(X,H^{\otimes m}).\]

The following is our main theorem.
\begin{thm}\label{thm: criterion}
 Let $X$ be a smooth projective variety of dimension $\geq 2$ and let $L$ be a globally generated ample line bundle on $X$.

 Assume that
 \begin{enumerate}
 \item $\rho(X)=1$ and take $H$ to be an ample generator of the N\'eron--Severi group of $X$;

 \item $-K_X$ is nef;
 \item
 the Hilbert polynomial $P_H(t)$ satisfies $P_H(1)>0$ and
\begin{align*}
 P_H(t) = \sum_{i=0}^n a_i t^i, \quad \text{with } a_i \geq 0 \text{ for } i \geq 2.
\end{align*}
 \end{enumerate}
Then the syzygy bundle $M_{L}$ is $\mu_H$-stable.
\end{thm}
\begin{rmk}
 In practice, condition (3) in Theorem~\ref{thm: criterion} can be replaced by the following stronger condition: \begin{enumerate}
 \item[(3')] $P_H(t)$ is a polynomial with non-negative coefficients.
 \end{enumerate}
 Indeed, as $H$ is ample, $P_H(t)$ is not identically zero, hence $P_H(1)>0$ which is the sum of all coefficients.
\end{rmk}

\subsection{Applications}
We apply Theorem~\ref{thm: criterion} to many varieties satisfying the criterion.

\begin{define}Let $X$ be a smooth projective variety.
\begin{enumerate}
 \item $X$ is a {\it weak Calabi--Yau} variety if $K_X\equiv 0$. 

 \item $X$ is a {\it hyperk\"ahler} variety if $X$ is simply connected and $\textrm{H}^0(X, \Omega^2_X)$ is spanned by an everywhere non-degenerate $2$-form.

\item $X$ is a {\it Fano} variety if $-K_X$ is ample.
\end{enumerate}

\end{define}

\begin{thm}[=Theorem~\ref{thm ci stable}+Corollary~\ref{cor other varieties}]
 Let $X$ be a smooth projective variety of dimension $\geq 2$ and let $L$ be a globally generated ample line bundle on $X$. Suppose that $X$ is one of the following:
\begin{enumerate}
\item a smooth complete intersection of dimension $\geq 3$ in a projective space such that $-K_X$ is nef;

\item a hyperk\"ahler variety of Picard number $1$;
\item an abelian variety of Picard number $1$;
\item a rational homogeneous variety of Picard number $1$;
\item a weak Calabi--Yau variety of Picard number $1$ of dimension $\leq 4$; or
\item a Fano variety of Picard number $1$ of dimension $\leq5$.
\end{enumerate}
 Then the syzygy bundle $M_{L}$ is $\mu_L$-stable.
\end{thm}
During the proof of the complete intersection case, we prove the following result, which might be of independent interest.
We could not find such a statement in literature.

\begin{thm}[see Theorem~\ref{thm ci poly>0}]
Let $X$ be a smooth complete intersection in a projective space such that $-K_X$ is nef (namely, $X$ is a Fano or Calabi--Yau complete intersection).
Then the Hilbert polynomial
$P_{\mathcal{O}_X(1)}(t)$ has non-negative coefficients.
\end{thm}

Parallel to the work of Caucci and Lahoz \cite{caucci2021stability} on abelian varieties, we prove the stability of syzygy bundles on any
 hyperk\"{a}hler variety by a similar strategy of \cite{caucci2021stability} with the help of a result of Rekuski.

 \begin{thm}\label{thm HK}
 Let $X$ be a hyperk\"{a}hler variety and let $L$ be a globally generated ample line bundle on $X$.
 Then the syzygy bundle $M_{L}$ is $\mu_L$-stable.
\end{thm}

 The organization of the paper is as the following: in Section~\ref{section 2}, we prove our criterion on stability of syzygy bundles on varieties of Picard number $1$; in Section~\ref{section 3}, we apply the criterion to study stability of syzygy bundles on Fano or Calabi--Yau complete intersections; in Section~\ref{section 4}, we apply the criterion to study stability of syzygy bundles on other varieties of Picard number $1$; in Section~\ref{section 5}, we prove the stability of syzygy bundles on
 hyperk\"{a}hler varieties.

\section{A stability criterion for syzygy bundles on varieties of Picard number one}\label{section 2}
In this section, we prove Theorem~\ref{thm: criterion}. The proof of Theorem~\ref{thm: criterion} uses the approach of Coand{\u a} \cite[Theorem 1.1]{coandua2011stability} and Ein--Lazarsfeld--Mustopa
 \cite[Proposition~C]{ein2013stability} with an additional analysis of Hilbert polynomials.

 The following lemma is essential for Theorem~\ref{thm: criterion}.
 \begin{lemma}\label{lem monotone}
 Let $X$ be a smooth projective variety of dimension $\geq 2$, and let $H$ be an ample line bundle on $X$ with $H\otimes \omega_X^{-1}$ ample.
 Suppose that the Hilbert polynomial $P_H(t)$ satisfies condition (3) of Theorem~\ref{thm: criterion}.
 Then \[\frac{P_H(k+1)-1}{k+1}\geq \frac{P_H(k)-1}{k}\] for any integer $k\geq 1$.
 Moreover, the equality holds only if $k=1$ and a general element in $|L|$ is reducible for any line bundle $L\equiv H^{\otimes 2}$.
\end{lemma}
\begin{proof}
Write $P_H(t) = \sum_{i=0}^n a_i t^i$ with $a_i\geq 0$ for $i\geq 2$. Here $n=\dim X\geq 2$ and $a_n>0$.


 Fix an ample line bundle  $L\equiv H^{\otimes 2}$. Then $H_1:=L\otimes H^{-1}\equiv H$. Note that $L\otimes \omega_X^{-1}$ and $H\otimes \omega_X^{-1}$ are both ample, then
 $P_H(2)=\mathrm{h}^0(X,H^{\otimes2})$ and
$$\mathrm{h}^0(X,H)=P_H(1)=P_{H_1}(1)=\mathrm{h}^0(X,H_1)$$
by the Kodaira vanishing. Here the Hilbert polynomial is independent of the numerical class of $H$ by the Riemann--Roch formula.
So by \cite[Lemma~15.6.2]{kollar1995shafarevich},
\[P_H(2)=\mathrm{h}^0(X,L)\geq \mathrm{h}^0(X,H)+\mathrm{h}^0(X,H_1)-1=2P_H(1)-1,\]
and the equality holds only if a general element in $|L|$ is reducible. This proves the assertion for $k=1$.
Moreover, this implies that
\begin{align}
    \sum_{i=2}^na_i(2^i-2)\geq a_0-1.\label{eq:1}
\end{align}

Then for integer $k\geq 2$,
 \begin{align*}
 {}& k(P_H(k+1)-1)-(k+1)( P_H(k)-1)\\
 ={}&\sum_{i=2}^na_i(k(k+1)^{i}-(k+1)k^{i})-(a_0-1) \\
 > {}& \sum_{i=2}^na_i(2^i-2)-(a_0-1)\geq 0.
 \end{align*}
 Here for the first inequality, we use the fact that
 \[f_i(k):=k(k+1)^{i}-(k+1)k^{i}=k(k+1)\sum_{j=0}^{i-2}\binom{i-1}{j}k^j\]
 is strictly increasing in $k$ and so $f_i(k)> f_i(1)=2^i-2$, while the last one is by \ref{eq:1}.
\end{proof}

Recall the following two key lemmas:

\begin{lemma}[{\cite[Lemma~2.1]{coandua2011stability}}]\label{lem slope 1}
Let $X$ be a smooth projective variety, let $H$ be an ample line bundle on $X$, and let $E$ be a vector bundle on $X$. If for every integer $r$ with $0<r<\mathrm{rk}(E)$ and for every line bundle $N$ on $X$ with $\mu_H(\bigwedge^rE\otimes N)\leq 0$ one has $\mathrm{H}^0(X, \bigwedge^rE\otimes N)=0$, then $E$ is $\mu_{H}$-stable.
\end{lemma}
\begin{lemma}[{\cite[Theorem~3.a.1]{green1984koszul}, \cite[Lemma~2.2]{coandua2011stability}}]\label{lem slope 2}
Let $X$ be a smooth projective variety and let $L,N$ be line bundles on $X$. Assume that $L$ is globally generated. Then $\mathrm{H}^0(X,\bigwedge^rM_L\otimes N)=0$ for $r\geq \mathrm{h}^0(X, N)$.
\end{lemma}
Now we can prove our main theorem.
\begin{proof}[Proof of Theorem~\ref{thm: criterion}]
As $\rho(X)=1$, every line bundle on $X$ is numerically equivalent to some $H^{\otimes k}$ for  $k\in \mathbb{Z}$. Suppose that $L\equiv H^{\otimes \ell}$ where $\ell\geq 1$.
Recall that $\textrm{rk}(M_L)=\mathrm{h}^0(X, L)-1$ and $c_1(M_L)=-c_1(L)=-\ell c_1(H)$.

Take $N\equiv H^{\otimes k}$ for some $k\in \mathbb{Z}$ and $0<r< \mathrm{h}^0(X, L)-1$ such that
$\mu_H(\bigwedge^rM_L\otimes N)\leq 0$.
By Lemma~\ref{lem slope 1}, it suffices to show that $\mathrm{H}^0(X,\bigwedge^rM_L\otimes N)=0$. By Lemma~\ref{lem slope 2}, it suffices to show that $r\geq \mathrm{h}^0(X,N)$. We may assume that $k>0$.

 Recall that for a vector bundle $E$ of rank $m$, \[c_1(\bigwedge^rE)=\binom{m-1}{r-1}c_1(E)\text{ and }\mathrm{rk}(\bigwedge^rE)=\binom{m}{r}.\]
Hence \[0\geq \mu_H(\bigwedge^rM_L\otimes N)= \left(k-\frac{\ell r}{\mathrm{h}^0(X,L)-1}\right)c_1(H)^{\dim X},\]
which implies
\begin{align}
    k\leq \frac{\ell r}{\mathrm{h}^0(X,L)-1}. \label{eq k<lr}
\end{align}
As $r< \mathrm{h}^0(X,L)-1$, we have $k<\ell $.
So by Lemma~\ref{lem monotone},
we have
\begin{align}\frac{P_H(\ell)-1}{\ell}\geq \frac{P_H(k)-1}{k}.\label{eq Pl>Pk}
\end{align}
We claim that the inequality is strict. If the equality holds, then  by Lemma~\ref{lem monotone}, we have $k=1, \ell =2$, and a general element in $|L|$ is reducible, but this contradicts Bertini's theorem as $L$ is globally generated ample.

Now as $L\equiv H^{\otimes \ell}$ and
$N\equiv H^{\otimes k}$, by the Kodaira vanishing and the Riemann--Roch formula, $\mathrm{h}^0(X,L)=P_H(\ell)$ and $\mathrm{h}^0(X,N)=P_H(k)$.
So combining \eqref{eq k<lr} and \eqref{eq Pl>Pk}, we have
\[r\geq \frac{k(\mathrm{h}^0(X,L)-1)}{\ell }=\frac{k(P_H(\ell)-1)}{\ell }>P_H(k)-1= \mathrm{h}^0(X,N)-1.\]
The proof is completed.
\end{proof}

\section{Stability of syzygy bundles on complete intersections}\label{section 3}
In this section, we prove the stability of syzygy bundles on Fano or Calabi--Yau complete intersections which is a generalization of \cite[Proposition 1.1]{coandua2011stability}.

\begin{thm}\label{thm ci stable}
Let $X$ be a smooth complete intersection in $\mathbb{P}^n$ of dimension  $\geq 3$ such that $-K_X$ is nef. Denote by $\mathcal{O}_X(d)=\mathcal{O}_{\mathbb{P}^n}(d)|_X$.
Then $M_{\mathcal{O}_X(d)}$ is $\mu_{\mathcal{O}_X(1)}$-stable for any $d\geq 1$.
\end{thm}

The proof of Theorem~\ref{thm ci stable} is by Theorem~\ref{thm: criterion} and the following theorem on the non-negativity of Hilbert polynomials of Fano or Calabi--Yau complete intersections.

\begin{thm}\label{thm ci poly>0}
 Let $n> k\geq 0$ be integers. Let $d_1, d_2, \dots, d_k$ be positive integers such that $\sum_{i=1}^kd_i\leq n+1$.
Let $X$ be a smooth complete intersection in $\mathbb{P}^n$ of multi-degree $(d_1, d_2, \dots, d_k)$. Denote by $\mathcal{O}_X(1)=\mathcal{O}_{\mathbb{P}^n}(1)|_X$. Then the Hilbert polynomial
$P_{\mathcal{O}_X(1)}(t)$ is a polynomial with non-negative coefficients.
\end{thm}

Before giving the proof of these theorems, we prove some basic properties on binomial coefficient polynomials.

Fix integers $n>0$ and $k\geq 0$. Let $d_1, d_2, \dots, d_k$ be positive integers.
Recall that the {\it binomial coefficient polynomial} $\binom{t}{n}$ in $t$ is defined by $\binom{t}{n}=\frac{1}{n!}\prod_{i=0}^{n-1}(t-i)$.

Denote ${\bf k}=\{1, 2, \dots, k\}$.
For a subset $I\subset {\bf k}$, denote $d_I=\sum_{i\in I}d_i$.
We define a polynomial as the following:
\begin{align*}
F_n(t; d_1, d_2, \dots, d_k): =\sum_{I\subset {\bf k}}(-1)^{|I|}\binom{t+n-d_I}{n}.
\end{align*}
For example, $F_n(t)  = \binom{t+n}{n}$ and
$F_n(t; d_1)  = \binom{t+n}{n}-\binom{t+n-d_1}{n}$.

\begin{lemma}\label{lemma 1.1}If $n\geq 2$ and $k\geq 1$, then
\[\sum_{I\subset {\bf k}}(-1)^{|I|}d_{{\bf k}\setminus I} \binom{t+n-1-d_I}{n-1}=\sum_{i=1}^kd_i F_{n-1}(t; d_1, \dots, \hat{d_i}, \dots, d_k).\]
Here the symbol $\hat{d_i}$ means to remove $d_i$ from the sequence.
\end{lemma}

\begin{proof}
\begin{align*}
{}&\sum_{I\subset {\bf k}}(-1)^{|I|}d_{{\bf k}\setminus I} \binom{t+n-1-d_I}{n-1}\\
={}& \sum_{I\subset {\bf k}}(-1)^{|I|}\sum_{i\in {\bf k}\setminus I}d_i \binom{t+n-1-d_I}{n-1}\\
={}&\sum_{i=1}^k d_i\sum_{I\subset {\bf k}\setminus\{i\}}(-1)^{|I|} \binom{t+n-1-d_I}{n-1}\\
={}&\sum_{i=1}^kd_i F_{n-1}(t; d_1, \dots, \hat{d_i}, \dots, d_k).
\end{align*}
\end{proof}

\begin{lemma}\label{lemma 1.2}If $n\geq 2$ and $k\geq 1$, then
\begin{align*}
{}& nF_n( t; d_1, d_2, \dots, d_k)\\
={}&(t+n-\sum_{i=1}^kd_i)F_{n-1}(t; d_1, d_2, \dots, d_k)+\sum_{i=1}^kd_i F_{n-1}(t; d_1, \dots, \hat{d_i}, \dots, d_k).
\end{align*}
\end{lemma}
\begin{proof}
\begin{align*}
{}&nF_n( t; d_1, d_2, \dots, d_k)\\
={}&\sum_{I\subset {\bf k}}(-1)^{|I|}n\binom{t+n-d_I}{n}\\
={}&\sum_{I\subset {\bf k}}(-1)^{|I|}(t+n-d_I)\binom{t+n-1-d_I}{n-1}\\
={}&(t+n-\sum_{i=1}^kd_i)\sum_{I\subset {\bf k}}(-1)^{|I|} \binom{t+n-1-d_I}{n-1}
+\sum_{I\subset {\bf k}}(-1)^{|I|}d_{{\bf k}\setminus I} \binom{t+n-1-d_I}{n-1}\\
={}&(t+n-\sum_{i=1}^kd_i)F_{n-1}(t; d_1, d_2, \dots, d_k)+\sum_{i=1}^kd_i F_{n-1}(t; d_1, \dots, \hat{d_i}, \dots, d_k).
\end{align*}
Here for the last step, we use Lemma~\ref{lemma 1.1}.
\end{proof}

Note that $\deg F_n( t; d_1, d_2, \dots, d_k)\leq n$.

\begin{lemma}\label{lem: n+1 Sj}For $0\leq j\leq n$, denote the coefficient of $t^{n-j}$ in $n!F_n( t; d_1, d_2, \dots, d_k)$ by $S_j(n; d_1, d_2, \dots, d_k)$.
If $\sum_{i=1}^kd_i=n+1$, then
\[
S_j(n; d_1, d_2, \dots, d_k)=\begin{cases} 2S_j(n; d_1, d_2, \dots, d_{k-1})& \text{if } k+j \text{ is even;}\\
0 & \text{if } k+j \text{ is odd.}\end{cases}\]
\end{lemma}
\begin{proof}
Denote $\sigma_j(t_1, t_2, \dots, t_n)$ to be the $j$-th elementary symmetric polynomial in $t_1, \dots, t_n$.
Note that the coefficient of $t^{n-j}$ in $n!\binom{t+n-m}{n}$ is just $\sigma_{j}(1-m, 2-m, \dots, n-m)$. So by the definition of $F_n( t; d_1, d_2, \dots, d_k)$,
\[
S_j(n; d_1, d_2, \dots, d_k)=\sum_{I\subset {\bf k}}(-1)^{|I|}\sigma_{j}(1-d_I, 2-d_I, \dots, n-d_I).
\]
By the assumption that $\sum_{i=1}^kd_i=n+1$, we have $d_I+d_{{\bf k}\setminus I}=n+1$.
So
\begin{align*}
{}&(-1)^{|{\bf k}\setminus I|}\sigma_{j}(1-d_{{\bf k}\setminus I}, 2-d_{{\bf k}\setminus I}, \dots, n-d_{{\bf k}\setminus I})\\
={}&(-1)^{|I|+k}\sigma_{j}(d_{I}-n, d_{I}-n+1, \dots, d_{I}-1)\\
={}&(-1)^{|I|+k+j}\sigma_{j}(1-d_{I}, 2-d_{I}, \dots, n-d_{I}).
\end{align*}
Then
\begin{align*}
{}&S_j(n; d_1, d_2, \dots, d_k)\\
={}&\sum_{I\subset {\bf k}}(-1)^{|I|}\sigma_{j}(1-d_I, 2-d_I, \dots, n-d_I)\\
={}&\sum_{I\subset {\bf k}\setminus\{k\}}(-1)^{|I|}\sigma_{j}(1-d_I, 2-d_I, \dots, n-d_I)+\sum_{k\in I\subset {\bf k}}(-1)^{|I|}\sigma_{j}(1-d_I, 2-d_I, \dots, n-d_I)\\
={}&\sum_{ I\subset {\bf k}\setminus\{k\}}(-1)^{|I|}\sigma_{j}(1-d_I, 2-d_I, \dots, n-d_I)\\{}&+ \sum_{{\bf k}\setminus I\subset {\bf k}\setminus\{k\}}(-1)^{|{\bf k}\setminus I|+j+k}\sigma_{j}(1-d_{{\bf k}\setminus I}, 2-d_{{\bf k}\setminus I}, \dots, n-d_{{\bf k}\setminus I})\\
={}&(1+(-1)^{j+k}) \sum_{ I\subset {\bf k}\setminus\{k\}}(-1)^{|I|}\sigma_{j}(1-d_I, 2-d_I, \dots, n-d_I)\\
={}&(1+(-1)^{j+k})S_j(n; d_1, d_2, \dots, d_{k-1}).
\end{align*}
\end{proof}

\begin{thm}\label{thm Fn>=0}
If $\sum_{i=1}^kd_i\leq n+1$, then
$F_n( t; d_1, d_2, \dots, d_k)$ is a polynomial in $t$ with non-negative coefficients.
\end{thm}
\begin{proof}
 We do induction on $n+k$.

 If $k=0$, then $F_n( t; d_1, d_2, \dots, d_k)=\binom{t+n}{n}$ has non-negative coefficients.

 If $n=1$ and $k>0$, then $k\leq 2$ and
\[F_1( t; d_1, d_2, \dots, d_k)=\begin{cases} d_1 & \text{ if } k=1;\\
0 & \text{ if } k=2.
\end{cases}\]

Now we consider $n>1$ and $k>0$. If $\sum_{i=1}^kd_i=n+1$, then by Lemma~\ref{lem: n+1 Sj}, we get the conclusion from the inductive hypothesis for $(n, k-1)$; if $\sum_{i=1}^kd_i\leq n$, then by Lemma~\ref{lemma 1.2}, we get the conclusion from the inductive hypothesis for $(n-1, k-1)$ and $(n-1, k)$.
\end{proof}

\begin{proof}[Proof of Theorem~\ref{thm ci poly>0}]
By an inductive argument by exact sequences (see \cite[Proposition~7.6]{GTM52}), it is well-known that the Hilbert polynomial
$P_{\mathcal{O}_X(1)}(t)$ is exactly $F_n(t; d_1, d_2, \dots, d_k)$, so the theorem follows from Theorem~\ref{thm Fn>=0}.
\end{proof}

\begin{proof}[Proof of Theorem~\ref{thm ci stable}]
Note that $\mathcal{O}_{X}(d)$ is globally generated ample for any $d\geq 1$.
Also $\mathrm{Pic}(X)= \mathbb{Z}[\mathcal{O}_{X}(1)]$ by the Lefschetz theorem inductively (\cite[Example~3.1.25]{positivity1}). Here one should be aware that $X=\bigcap_{i=1}^kX_i$ is a smooth intersection of hypersurfaces $X_i\subset \mathbb{P}^n$ of degree $d_i$ for $1\leq i\leq k$, but $X_i$ might not be smooth. However we can deform $X_i$ into smooth hypersurfaces to get the conclusion as in this case
$\mathrm{Pic}(X)\cong \mathrm{H}^2(X, \mathbb{Z})$
is invariant under deformation.

Then by Theorem~\ref{thm ci poly>0}, all conditions in Theorem~\ref{thm: criterion} are satisfied, and hence we conclude the theorem.
\end{proof}

\section{Stability of syzygy bundles on other varieties of Picard number one}\label{section 4}

Besides complete intersections, there are also other varieties satisfying Theorem~\ref{thm: criterion}.

\begin{cor} \label{cor other varieties} Let $X$ be a smooth projective variety of dimension $\geq 2$ and let $L$ be a globally generated ample line bundle on $X$. Suppose that $X$ is one of the following:
\begin{enumerate}
\item a hyperk\"ahler variety of Picard number $1$;
\item an abelian variety of Picard number $1$;
\item a rational homogeneous variety of Picard number $1$;
\item a weak Calabi--Yau variety of Picard number $1$ of dimension $\leq 4$; or
\item a Fano variety of Picard number $1$ of dimension $\leq5$.
\end{enumerate}
 Then $M_{L}$ is $\mu_L$-stable.
\end{cor}
\begin{proof}
 Set $n=\dim X$. Take $H$ to be an ample generator of the N\'eron--Severi group of $X$.

In all cases, $-K_X$ is nef.
Note that $\mu_L$-stability and $\mu_H$-stability are equivalent.
So by Theorem~\ref{thm: criterion}, it suffices to check that condition (3) in
Theorem~\ref{thm: criterion} holds in each case. By the Lefschetz principle, we may assume that $X$ is defined over $\mathbb{C}$.

 In case (1), $P_H(t)$ has non-negative coefficients by \cite[Theorem~1.1]{jiang2023positivity}.

In case (2), $P_H(t)=\frac{H^n}{n!}t^n$.

 In case (3), we may write $X=G/P$ where $G$ is a simple Lie group and $P$ is a maximal parabolic subgroup. 
 Then $H$ corresponds to a weight $\lambda$ and 
 by the Borel--Bott--Weil theorem, \[
 P_H(t)=\chi(X,H^{\otimes t})=\prod_{\alpha\in \Phi^+}\left(t\frac{\left<\lambda,\alpha\right>}{\left<\rho,\alpha\right>}+1\right)\]
 is a product of linear polynomials with non-negative coefficients by \cite{Gross2011}. Here $\rho$ is half the sum of the positive roots.
 Then the result follows.

 Finally we consider cases (4) and (5).
 By the Riemann--Roch formula, the Hilbert polynomial is given by
 \[P_H(t)=\chi(X,H^{\otimes t})=\sum_{i=0}^n \frac{\textrm{td}_{n-i}(X)H^i}{i!}t^i. \]
Recall that $\textrm{td}_{0}(X)=1$,
$\textrm{td}_{1}(X)=\frac{1}{2} c_1(X)$, $\textrm{td}_{2}(X)=\frac{1}{12}(c_1(X)^2+c_2(X))$, $\textrm{td}_{3}(X)=\frac{1}{24}c_1(X)c_2(X)$, and $\textrm{td}_{n}(X)=\chi(X, \mathcal{O}_X)$.

We claim that if $X$ is a Fano variety or a weak Calabi--Yau variety, then the constant term $\chi(X, \mathcal{O}_X)\geq 0$. Indeed, if $X$ is a Fano variety, then $\chi(X, \mathcal{O}_X)=1$ by the Kodaira vanishing. On the other hand, if $X$ is a weak Calabi--Yau variety, then by the Beauville--Bogomolov decomposition \cite{beauville1983varietes, bogomolov}, there is a finite \'etale cover $\pi: X'\to X$ such that
\[
X' \cong A\times X_1\times \cdots \times X_m\times Y_1\times \cdots \times Y_k,
\]
where $A$ is an abelian variety, $X_i$ is a Calabi--Yau variety, and $Y_j$ is a hyperk\"{a}hler variety. Note that $\chi(A, \mathcal{O}_A)=0$ if $A$ is not a point, $\chi(X_i, \mathcal{O}_{X_i})=1+(-1)^{\dim X_i}$, and $\chi(Y_j, \mathcal{O}_{Y_j})=\frac{\dim Y_j}{2}+1$, so $\chi(X, \mathcal{O}_X)=\frac{1}{\deg \pi}\chi(X', \mathcal{O}_{X'})\geq 0$.

Since $-K_X=c_1(X)$ is nef, the intersection of $c_2(X)$ with nef line bundles are non-negative by \cite[Corollary~1.5]{Ou23}, so
$\textrm{td}_{n-i}(X)H^i\geq 0$ for $0\leq n-i\leq 3$.

So if $n\leq 4$, then $P_H(t)$ has non-negative coefficients and the proof is completed.

If $n=5$, then all the coefficients of $P_H(t)$ are non-negative except for the coefficient of $t$. But we can still apply Theorem~\ref{thm: criterion} as $P_H(1)=\mathrm{h}^0(X, H)>0$ by \cite[Corollary~1.3]{Horingfano5}.
\end{proof}

We might not expect that for any smooth Fano or weak Calabi--Yau variety of Picard number $1$, $P_H(t)$ has non-negative coefficients, where $H$ is the ample generator of the N\'eron--Severi group. But anyway we can ask the following question.
\begin{question}\label{question 1}
 Let $X$ be a smooth Fano or weak Calabi--Yau variety of Picard number $1$ and let $L$ be a globally generated ample line bundle. Is $M_L$ $\mu_L$-stable?
\end{question}
Questions~\ref{question 1} is known for sufficiently ample $L$ by \cite{rekuski2024stability}.

\section{Stability of syzygy bundles on hyperk\"{a}hler varieties}\label{section 5}

In this section, we give the proof of Theorem~\ref{thm HK}. We use a similar approach to that in \cite[Theorem~1.1]{caucci2021stability}. The proof of \cite[Theorem~1.1]{caucci2021stability} relies on the density of simple abelian varieties in the moduli space of abelian varieties, while we use the fact that hyperk\"{a}hler varieties of Picard number 1 are dense in the moduli space of polarized hyperk\"{a}hler varieties instead.

\begin{proof}[Proof of Theorem~\ref{thm HK}]
By the Lefschetz principle, we may assume that $X$ is defined over $\mathbb{C}$.
By the proof of \cite[Theorem~4.6]{huybrechts1999compact}, we can
consider $\mathcal{X}\to S$ a family of hyperk\"{a}hler varieties over a smooth curve polarized by a relative ample line bundle $\mathcal{L}$, such that $(\mathcal{X},\mathcal{L})_0\cong (X,L)$ for $0\in S$ and
\[S^\circ=:\{s\in S \mid \rho(\mathcal{X}_s)=1 \}\]
is an open dense subset in $S$.

Since $L$ is globally generated, we may assume that $\mathcal{L}_s$ is globally generated for all $s\in S$, up to shrinking $S$ if needed. By Corollary~\ref{cor other varieties},
for any $s\in S^\circ$, the syzygy bundle $M_{\mathcal{L}_s}$ is $\mu_{\mathcal{L}_s}$-stable.

Denote $P(m):=\chi(M_{\mathcal{L}_s}\otimes \mathcal{L}_s^{\otimes m})$ to be the Hilbert polynomial of $M_{\mathcal{L}_s}$, which is independent of $s\in S^{\circ}$ by shrinking $S$.
By \cite[Theorem~4.3.7]{huybrechts2010geometry}, there exists a projective morphism $\mathsf{M}_{\mathcal{X}/S}(P)\to S$
universally corresponding to the moduli functor $\mathcal{M}_{\mathcal{X}/S}(P)\to S$ of semistable sheaves with Hilbert polynomial $P$. Since
$S^\circ$ is dense in $S$, by the properness of the moduli space, there exists a family $\mathcal{F}\in \mathsf{M}_{\mathcal{X}/S}(P)$ such that $\mathcal{F}_s=M_{\mathcal{L}_s}$ for all $s\in S^\circ$. Note that $\mathcal{F}_0$ is a $\mu_L$-semistable
torsion-free sheaf with $c_1(\mathcal{F}_0)=c_1(\mathcal{F}_s)=-c_1(L)$. We will prove that $ \mathcal{F}_0^{\vee\vee}\cong M_{L}$.

By the semicontinuity, for $s\in S^\circ$, we have \[\dim \mathrm{Hom}(\mathcal{F}_0,\mathcal{O}_X)\geq \mathrm{h}^0(\mathcal{X}_s,\mathcal{F}_s^\vee)= \mathrm{h}^0(\mathcal{X}_s,\mathcal{L}_s)=\mathrm{h}^0(X,L).\]
Here for the first equality, we use the short exact sequence
\[
0\to \mathcal{L}_s^\vee\to \mathrm{H}^0(\mathcal{X}_s,\mathcal{L}_s)\otimes \mathcal{O}_{\mathcal{X}_s}\to \mathcal{F}_s^\vee\to 0
\]
and $\mathrm{h}^0(\mathcal{X}_s,\mathcal{L}_s^\vee)=\mathrm{h}^1(\mathcal{X}_s,\mathcal{L}_s^\vee)=0$ by the Kodaira vanishing.

Recall that $\textrm{rk}(\mathcal{F}_0)=\mathrm{H}^0(X,L)-1$. Then we can consider a map $u: \mathcal{F}_0\to \mathrm{H}^0(X,L )\otimes \mathcal{O}_X$ which is generically of maximal rank, and hence injective as $\mathcal{F}_0$ is torsion-free. Consider the commutative diagram
\[\xymatrix{
0\ar[r]&\mathcal{F}_0\ar[r]^-{u}\ar@{^{(}->}[d]^{f} &\mathrm{H}^0(X,L )\otimes \mathcal{O}_X\ar[r]\ar@{=}[d]&Q\ar[r]\ar[d]^g&0\\
0\ar[r]&\mathcal{F}_0^{\vee\vee}\ar[r]
&\mathrm{H}^0(X,L )\otimes \mathcal{O}_X\ar[r]^-v &Q^{\vee\vee}.&
}
\]
Consider short exact sequences
\begin{align*}
 0\to{}& \mathrm{Ker} (g)\to Q\to \mathrm{Im}(g)\to 0;\\
 0\to{}&\mathrm{Im}(g)^\vee \to Q^\vee\to \mathrm{Ker} (g)^\vee.
\end{align*}
Note that $\mathrm{Ker}(g)\cong \mathrm{Coker}(f)$ has codimension $\geq 2$, so $c_1(\mathrm{Ker}(g))=0$ and $\mathrm{Ker}(g)^\vee=0$. In particular, $\mathrm{Im}(g)^\vee \cong Q^\vee$.
As $\mathrm{Im}(g)$ is torsion-free, by the above exact sequences, we have
\begin{align*}
c_1(Q^{\vee\vee})=c_1(\mathrm{Im}(g)^{\vee\vee})=c_1(\mathrm{Im}(g))=c_1(Q)=-c_1(\mathcal{F}_0)=c_1(L).
\end{align*}
Since $\mathrm{H}^1(X,\mathcal{O}_X)=0$, $c_1$ is injective on the Picard group. Thus $Q^{\vee\vee}= L.$

We claim that $v$ is surjective. Suppose, to the contrary, that it is not surjective. Then
$\mathrm{Im}(v)= I_Z\otimes L $ for some proper closed subscheme $Z$ of $X$.
So we have the following commutative diagram
\[\xymatrix{
0\ar[r]&\mathcal{F}_0^{\vee\vee}\ar[r]
&\mathrm{H}^0(X,L )\otimes \mathcal{O}_X\ar[r]^-{v}\ar[d]^{p} &I_Z\otimes L \ar[r]\ar@{^{(}->}[d]&0\\
&
&\mathrm{H}^0(X,L )\otimes \mathcal{O}_X\ar[r]^-{\textrm{ev}} & L \ar[r]&0.
}\]
Then $\mathcal{F}_0^{\vee\vee}$ has a non-zero subsheaf isomorphic to $\mathrm{Ker}(p)$ which is of the form $W\otimes \mathcal{O}_X$.
But this contradicts the fact that $\mathcal{F}_0^{\vee\vee}$ is $\mu_L$-semistable with $\mu_L(\mathcal{F}_0^{\vee\vee})=\mu_L(\mathcal{F}_0)<0$.

Therefore, $v$ coincides with the evaluation map and $M_{L}=\mathrm{Ker} (\textrm{ev} )\cong \mathcal{F}_0^{\vee\vee}$ is $\mu_L$-semistable. Finally, by \cite[Theorem~4.2.8]{rekuski2022wall}, $M_L$ is $\mu_L$-stable.
\end{proof}
\begin{rmk}
By \cite[Theorem~4.2.8]{rekuski2022wall}, the $\mu_L$-semistability in \cite[Theorem 1.1]{caucci2021stability} can also be strengthened to $\mu_L$-stability. The usage of \cite[Theorem~4.2.8]{rekuski2022wall} was pointed to us by Federico Caucci.
\end{rmk}

\section*{Acknowledgments}
This work was supported by National Key Research and Development Program of China \#2023YFA1010600, \#2020YFA0713200, and NSFC for Innovative Research Groups \#12121001. The authors are members of the Key Laboratory of Mathematics for Nonlinear Sciences, Fudan University. 
The authors are grateful to Yalong Cao, Federico Caucci, Rong Du, Hanfei Guo, Zhiyuan Li, Nick Rekuski,  Yang Zhou, and the referee for their valuable discussions and suggestions.

 \bibliographystyle{abbrv}

\bibliography{ref}

\end{document}